\documentclass{amsart}
\usepackage{amsmath, amsthm, amssymb, enumitem, graphicx, color, hyperref}
\allowdisplaybreaks

\newtheorem{theorem}{Theorem}
\newtheorem{corollary}{Corollary}
\newtheorem{definition}{Definition}
\newtheorem{lemma}{Lemma}

\begin{document}
	\author{Isabel Beach}
	\title[Short Simple Orthogonal Geodesic Chords]{Short Simple Orthogonal Geodesic Chords on a 2-Disk with Convex Boundary}
	
	\begin{abstract}
		We prove the existence of at least two distinct short, simple orthogonal geodesic chords on a Riemannian 2-disk $M$ with convex boundary. The lengths of these curves are bounded in terms of the length of $\partial M$, the diameter of $M$, and the area of $M$. We also prove the existence of a short, simple geodesic chord on any Riemannian 2-disk.
	\end{abstract}
	
	\maketitle
		
		\section{Introduction}
		
		In this paper we investigate the lengths of orthogonal geodesic chords on a Riemannian manifold with boundary, denoted by $M$. These are geodesic segments that lie in the interior of $M$ except at their endpoints, where the segments meet $\partial M$ orthogonally. One motivation for the study of these curves is their relationship with certain periodic solutions of Hamiltonian systems called ``brake orbits" (see, for example, \cite{giambo2004}). Each brake orbit has an associated ``energy'', and brake orbits with energy $E$ lie in a certain subset $\Omega_E\subset M$. By the Maupertuis principle, $\Omega_E$ can be endowed with a particular metric $g_E$, known as the Jacobi metric, such that brake orbits correspond to geodesics in $(\Omega_E, g_E)$ with endpoints on $\partial \Omega_E$.
		Unfortunately, the Jacobi metric is not Riemannian on $\partial \Omega_E$, which makes such geodesics difficult to study. 
		However, in \cite{giambo2004}, R. Giamb\`o, F. Giannoni and P. Piccione proved that, under certain mild assumptions, there is a slightly smaller subset $\Omega_E'\subset\Omega_E$ such that any orthogonal geodesic chord in the Riemannian manifold $(\Omega_E', g_E)$ can be extended to produce a unique brake orbit with energy $E$. Therefore any issues caused by the degenerate metric can be avoided.
		In \cite{giambo2010}, Giamb\`o, Giannoni and Piccione proved that there exists at least one orthogonal geodesic chord on any $n$-disk with strongly concave boundary. In such circumstances, the existence of at least two orthogonal geodesic chords was later proven by the same authors in \cite{giambo2015}, and the existence of at least $n$ orthogonal geodesic chords was proven in \cite{giambo2020}. Moreover, they proved that these $n$ chords give rise to $n$ brake orbits in the above context.
		\par 
		Alternatively, one can consider chords on a manifold with convex boundary, meaning that any two sufficiently close points on the boundary are connected by a unique minimizing geodesic lying in the interior of the manifold. The existence of $n$ orthogonal geodesic chords on an $n$-disk with convex boundary was proven by W. Bos \cite{bos1963}. H. Gluck and W. Ziller proved in \cite{gluck1984} that every $n$-dimensional manifold with convex boundary admits at least one orthogonal geodesic chord. Moreover, on a 2-disk with convex boundary, J. Hass and P. Scott \cite{hass1994} and D. Ko \cite{ko2023} proved the existence of two \textit{simple} orthogonal geodesic chords.
		\par 
		In this article, we are interested in the quantitative study of orthogonal geodesic chords. The quantitative study of geodesics in general provides a wealth of techniques for us to build upon. The existence of infinitely many geodesic segments connecting any pair of points on a Riemannian manifold was proven by J.-P. Serre in \cite{serre1951}. On a compact manifold with diameter $d$, the existence of $k$ geodesic segments connecting any pair of points of length at most $20kd$ was proven by A. Nabutovsky and R. Rotman in \cite{rotman_linear_bounds}. This bound was improved to $6kd$ by H. Y. Cheng \cite{cheng2022}. These results also apply to geodesic loops. S. Sabourau proved in \cite{sabourau_2004_loops} that the length of a shortest geodesic loop on a complete manifold is bounded both by a function of the volume and by a function of the diameter. The existence of two simple geodesic loops at any point on a 2-sphere was proven by the author in \cite{beach2024}. This work builds on the techniques of Y. Liokumovich, Nabutovsky and Rotman in \cite{rotman_ls_2017}, who devised a quantitative version of the Lyusternik--Schnirelmann theorem of the existence of three simple closed geodesics on a 2-sphere. It also utilized the curve shortening flows developed by Hass and Scott in \cite{hass1994}. 
		\par 
		In this article, we modify the above techniques to find length bounds for orthogonal geodesic chords on a 2-disk. Altering these techniques to apply to curves with endpoints on the boundary of a manifold (as opposed to closed curves or loops with a fixed base point) is a novel contribution to the field. Moreover, the author is not aware of any published results regarding the quantitative properties of orthogonal geodesic chords, and as such the following theorem is a new result in understanding the behaviour of these curves. 
		\begin{theorem}
			\label{theorem:main}
			Let $M$ be a Riemannian $2$-disk with convex boundary, diameter $d$, boundary length $P$, and area $A$. Then at least one of the following holds.
			\begin{enumerate}
				\item 
				$M$ admits three simple orthogonal geodesic chords of index zero and length at most $2d$, and one simple orthogonal geodesic chord of positive index and length at most $2d+2P+686\sqrt{A}$.
				\item 
				$M$ admits one simple orthogonal geodesic chord of index zero and length at most $2d$, one simple orthogonal geodesic chord of positive index and length at most $2d+2P+686\sqrt{A}$, and $k$ non-simple orthogonal geodesic chords of index zero and lengths at most $4d+k(2d+P)$ for any integer $k\geq 2$.
				\item 
				$M$ admits two simple orthogonal geodesic chords of positive index and respective lengths at most $4d+P$ and $6d+2P$.
				\item 
				$M$ admits one simple orthogonal geodesic chord of index zero and length at most $2d$, and two simple orthogonal geodesic chords of positive index and respective lengths at most $8d+2P$ and $14d+4P$.
			\end{enumerate}
		\end{theorem}
		\noindent
		The following two corollaries are immediate by considering the worst possible bounds overall for the first and second shortest orthogonal geodesic chords in the above theorem.
		\begin{corollary}
			\label{cor:simple}
			Let $M$ be a Riemannian $2$-disk with convex boundary, diameter $d$, boundary length $P$, and area $A$. Then $M$ admits a simple orthogonal geodesic chord of length at most $4d+P$, and a second distinct simple orthogonal geodesic chord of length at most $2d+2P+686\sqrt{A}$.
		\end{corollary}
		\begin{corollary}
			\label{cor:non_simple}
			Let $M$ be a Riemannian $2$-disk with convex boundary, diameter $d$, boundary length $P$, and area $A$. Then $M$ admits an orthogonal geodesic chord of length at most $4d+P$, and a second distinct orthogonal geodesic chord of length at most $8d+2P$.
		\end{corollary}
		The length bounds in Theorem \ref{theorem:main} are optimal in the following sense. The chords with positive index are found by applying a min-max procedure to a sweepout of $M$ created by contracting the boundary of $M$ through short curves. Not all 2-disks have boundaries that can be contracted through curves with lengths bounded only in terms of the disk's diameter and the length of its boundary. For example, S. Frankel and M. Katz provide a sequence of metrics on the Riemannian 2-disk with diameter and perimeter both equal to one, but such that the disk's boundary can only be contracted through curves with arbitrarily long maximum length \cite{frankel_katz}.
		Thus, the fact that some of our bounds additionally depend on area is seemingly necessary. In some cases, we also obtain chords of index zero and length at most $2d$. We do not expect this to be possible in all cases. F. Balacheff, C. Croke and M. Katz provide an example in \cite{balacheff} of a 2-sphere whose shortest closed geodesic is strictly longer than $2d$. It seems plausible that a similar example would hold for a shortest orthogonal geodesic chord on a 2-disk.
		\par
		The proof of the above theorem is summarized as follows. Let $\Omega_{\partial M}M$ be the space of unparameterized curves in $M$ with endpoints on $\partial M$. We are interested in the subset $\Sigma_{\partial M} M\subset \Omega_{\partial M}M$ of curves without transverse self-intersections and the subset $\Sigma^0_{\partial M} M\subset \Sigma_{\partial M}M$ of point curves on $\partial M$. The space $(\Sigma_{\partial M} M,\Sigma^0_{\partial M} M)$ has a non-trivial homology class in dimensions one and two. We want to construct representatives for these homology classes whose images consist only of curves of bounded length. Following \cite{hass1994} and \cite{ko2023}, we can then apply a free boundary curve shortening flow to the curves in both images. A subsequence of curves in each image will converge to an orthogonal geodesic chord whose length is no longer than the longest curve in the image of the original representative. If the chords we obtain from the two representatives are equal, we will be able to use a standard Lyusternik--Schnirelmann argument to prove that there are in fact infinitely many simple orthogonal geodesic chords of the same length.
		\par 
		In order to construct these representatives, we will first form what we call a ``radial sweepout" of $M$ as per the following two definitions. 
		\begin{definition}
			\label{def:monotone}
			We call a family of curves (non-strictly) monotone if no curve in the family has transverse self-intersections and no two distinct curves in the family intersect transversely.
		\end{definition}
		\begin{definition}
			\label{def:sweepout}
			A radial sweepout of $M$ is a monotone family of curves $\Gamma_t$ parameterized by $S^1$ such that for all $t\in S^1$, $\Gamma_t(0)=p$ for some $p\in \operatorname{int} M$ and $\Gamma_t(1)\in\partial M$. Moreover, this map is homotopically non-trivial relative to $\partial M$.
		\end{definition}
		\noindent 
		An example of a radial sweepout of a 2-disk $M$ is depicted in Figure \ref{fig:example_sweepout}.
		\begin{figure}[ht]
			\centering
			\includegraphics[width=0.35\linewidth]{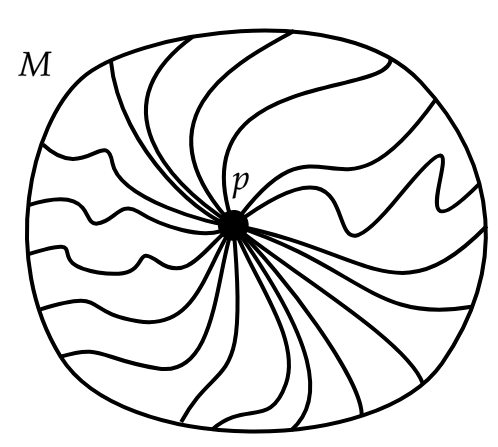}
			\caption{An example of a radial sweepout of a 2-disk.}
			\label{fig:example_sweepout}
		\end{figure}
		We construct this sweepout as follows. First, we will use a version of Berger's lemma (see Lemma \ref{lemma:berger})
		to divide $M$ into two or three regions bounded by pairs of minimizing geodesics that start at some $p$ and meet $\partial M$ orthogonally. We then apply the free boundary curve shortening flow to each pair of bounding geodesics. We would like each pair to contract to a point under this flow so that we can combine the three homotopies into a sweepout of $M$. If instead a subsequence of the resulting curves converges to an orthogonal geodesic chord, we can instead contract the pair to a point using the techniques of \cite{rotman_monotone_2017}, \cite{chambers_rotman_2013}, and \cite{lio_2015}. An orthogonal geodesic chord obtained by shortening can be no longer than the pair of minimizing geodesics and hence no longer than twice the diameter of $M$. In this sense, short simple orthogonal geodesic chords obstruct our ability to form a sweepout through short simple curves with endpoints on $\partial M$. 
		\par 
		After we have contracted all pairs of bounding geodesics, these contracting homotopies cover $M$ and can be combined into a radial sweepout. Roughly speaking, a curve in our radial sweepout is given by a curve in one of these homotopies followed by a short curve joining it to $p$. Our representative of the one-dimensional homology class will be this sweepout. The second representative will essentially be formed by taking all pairs of curves in the sweepout. We can then apply the Lyusternik--Schnirelmann proof to obtain two distinct simple orthogonal geodesic chords of bounded length.
		\par 
		Once the convex case is proven, we will be able to prove a weaker result in the non-convex case. Following the work of H. Seifert \cite{seifert_1949} and Gluck and Ziller \cite{gluck1984}, we affix a collar with convex boundary to $\partial M$. The resulting manifold satisfies the conditions of Theorem \ref{theorem:main} and thus admits a short orthogonal geodesic chord. Taking an arc of this chord that lies inside $M$, we obtain the following.
		\begin{theorem}
			\label{theorem:main_concave}
			Let $M$ be a Riemannian $2$-disk with diameter $d$ and boundary length $P$. Then $M$ admits a geodesic segment which lies in the interior of $M$ except at its endpoints, at least one of which meets $\partial M$ orthogonally. Moreover, this curve has length at most $4d+P$.
		\end{theorem}
		\noindent
		Note that this curve may not be an orthogonal geodesic chord, as it may not be orthogonal to $\partial M$ at both endpoints.
		\par 
		This article is organized as follows. We first recall the properties of the curve shortening flows we will be using in this paper in Section \ref{sec:flow}. We summarize the properties of the disk flow of \cite{hass1994} for closed curves (Theorem \ref{theorem:disk_flow}) and the free boundary disk flow for curves with endpoints on $\partial M$ (Theorem \ref{theorem:free_boundary_flow}). We then recall the properties of the free boundary flow on families of curves (Theorem \ref{theorem:family_shortening}). Lastly, we prove that convex curves (Definition \ref{def:convex}) behave well under these flows (Lemma \ref{lemma:shortening_monotonicity}).
		In Section \ref{sec:family}, we describe how to use curve shortening and a generalized version of Berger's Lemma (Lemma \ref{lemma:berger}) to construct a radial sweepout of $M$ through curves of bounded length in the absence of short orthogonal geodesic chords (Lemma \ref{lemma:main}). We also recall the result of \cite{lio_2015} which shows how to create such a sweepout even in the presence of short orthogonal geodesic chords, although with a much larger length bound (Theorem \ref{theorem:main_alt}).
		In Section \ref{sec:tighten}, we apply the standard Lyusternik--Schnirelmann argument to produce a short orthogonal geodesic chord from our radial sweepout (Lemma \ref{lemma:pull_tight}), using an intermediate lemma from \cite{hass1994} (Lemma \ref{lemma:shorten_or_nbhd}). We then prove Theorem \ref{theorem:main} by analyzing the cases of Lemma \ref{lemma:main} and, if necessary, applying Theorem \ref{theorem:main_alt}.
		Finally, in Section \ref{sec:concave}, we use our results for manifolds with convex bound to prove Theorem \ref{theorem:main_concave}.
		\\
		\ 
		\\
		\paragraph{\textbf{Acknowledgements.}} The author would like to thank Regina Rotman for her invaluable guidance. This work was supported in part by an NSERC Canada Graduate Scholarships Doctoral grant. This paper was partially written during the author's stay at the Summer Scholars program at the Institute for Advanced Study. It was largely completed while the author was in residence at the Simons Laufer Mathematical Sciences Institute (formerly MSRI) in Berkeley, California, during the Fall 2024 semester, supported by the National Science Foundation under Grant Number DMS-1928930.
		\section{Results}
		
		\subsection{The Disk Flow and the Free Boundary Disk Flow}
		\label{sec:flow}
		
		The first curve shortening process we will use is the disk flow of Hass and Scott \cite{hass1994}, which is defined on closed curves as follows. Suppose we wish to apply the disk flow to the closed curve $\gamma=\gamma_0$ to produce a family of curves $\gamma_t$ for $t\in[0,\infty)$. We first cover $M$ by totally normal metric balls $\{B_i\}_{i=1}^k$ such that the metric balls with half the radius and the same centres still cover $M$. The flow is defined in each $B_i$ consecutively, beginning with $B_1$. At step $i$, every arc of $\gamma\cap B_i$ with endpoints on $\partial B_i$ is replaced by a minimizing geodesic connecting its endpoints. Because $B_i$ is totally normal, this geodesic is unique and lies in $B_i$. If $\gamma$ lies entirely within $B_i$, we contract it to a point. After doing this in every ball, we obtain the curve $\gamma_1$. We define $\gamma_t$ for $t\in(0,1)$ as a homotopy between $\gamma_0$ and $\gamma_1$ through curves $\gamma_t$ that have no more transverse self-intersections than $\gamma_0$ has (see Lemma 1.6 of \cite{hass1994}). We iterate this process to define $\gamma_t$ for $t\in[i,i+1]$, $i\in\mathbb{N}$. 
		\par 
		We recall a few properties of the disk flow.
		\begin{theorem}[see Theorem 1.8 of \cite{hass1994}]
			\label{theorem:disk_flow}
			Let $\gamma_t$, $t\in[0,\infty)$, be the image of a closed curve $\gamma_0$ under the disk flow. Then the following hold.
			\begin{enumerate}
				\item 
				$L(\gamma_i)\leq L(\gamma_{i-1})$ for all $i\in\mathbb{N}$.
				\item 
				The number of transverse self-intersections of $\gamma_t$ is non-increasing as a function of $t$.
				\item 
				A subsequence $\{\gamma_{i_j}\}_{j\in\mathbb{N}}$ of the curves $\{\gamma_i\}_{i\in\mathbb{N}}$ converges to some $\gamma_\infty$, which is either a closed geodesic of length at most $\lim\limits_{j\to\infty} L(\gamma_{i_j})$ or a point curve.
				\item
				$0<L(\gamma_i)= L(\gamma_{i-1})$ if and only if $\gamma_i$ is a closed geodesic. \qed
			\end{enumerate}
		\end{theorem}
		\par
		The second process we will consider is a free boundary version of the disk flow, defined on curves with endpoints on $\partial M$. This flow is described briefly in Section 5 of \cite{hass1994}, but we will detail it here for completeness.
		First, cover $M$ by totally normal metric balls as in the disk flow. However, we now also ensure that the metric balls in our cover are small enough that if $B_i\cap \partial M\not=\emptyset$, then any $x\in B_i$ has a unique minimizing geodesic in $B_i$ connecting it to $\partial M$ that meets $\partial M$ orthogonally. Let $\gamma_0$ be the curve that we wish to shorten. As before, we act in each ball consecutively. An arc of $\gamma_0\cap B_i$ that does not contain any endpoints of $\gamma_0$ is replaced by the corresponding minimizing geodesic, as in the disk flow. An arc $\eta$ of $\gamma_0\cap B_i$ containing an endpoint of $\gamma_0$ is replaced by the unique minimizing geodesic connecting the point $\eta\cap \partial B_i$ to the nearest point of $\partial M\cap\partial B_i$. After doing this in each ball, we obtain $\gamma_1$. We then define $\gamma_t$ for $t\in(0,1)$ to be a homotopy between $\gamma_0$ and $\gamma_1$ that does not increase the number of transverse self-intersections, as in the disk flow. This completes the definition of the free boundary curve shortening flow. An example of this process is shown in Figure \ref{fig:csf}.
		\begin{figure}[ht]
			\centering
			\includegraphics[width=0.8\textwidth]{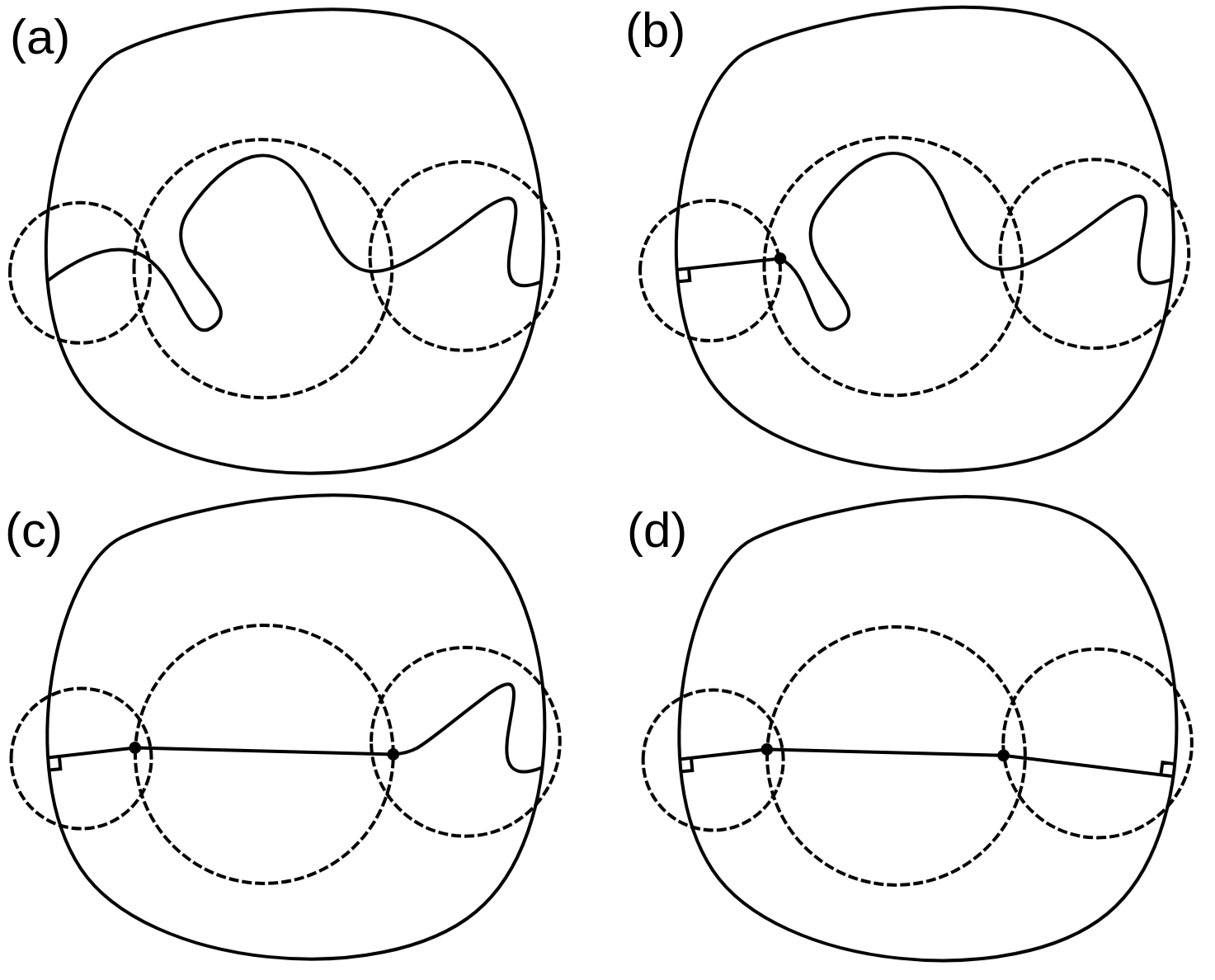}
			\caption{An illustrative example of the free boundary disk flow. (a) The initial curve covered by totally normal metric balls. (b) Straightening the arc in the first disk. (c) Straightening the arc in the second disk. (d) Straightening the arc in the third and final disk.}
			\label{fig:csf}
		\end{figure}
		\par 
		In analogy with Theorem \ref{theorem:disk_flow}, the free boundary disk flow has the following properties. The proof of this theorem is almost identical to that of Theorem 1.8 of \cite{hass1994}.
		\begin{theorem}
			\label{theorem:free_boundary_flow}
			Let $\gamma_0$ be a curve in $M$ with endpoints on $\partial M$. Let $\gamma_t$, $t\in[0,\infty)$, be the image of $\gamma_0$ under the free boundary disk flow. Then the following hold.
			\begin{enumerate}
				\item 
				$L(\gamma_i)\leq L(\gamma_{i-1})$ for all $i\in\mathbb{N}$.
				\item 
				The number of transverse self-intersections of $\gamma_t$ is non-increasing as function of $t$.
				\item         
				A subsequence $\{\gamma_{i_j}\}_{j\in\mathbb{N}}$ of the curves $\{\gamma_i\}_{i\in\mathbb{N}}$ converges to some $\gamma_\infty$, which is either an orthogonal geodesic chord of length at most $\lim\limits_{j\to\infty} L(\gamma_{i_j})$ or a point curve.
				\item
				$0<L(\gamma_i)= L(\gamma_{i-1})$ if and only if $\gamma_i$ is an orthogonal geodesic chord. \qed
			\end{enumerate}
		\end{theorem}
		\noindent  
		Note that if $\gamma_0$ has no transverse self-intersections, then by the above $\gamma_\infty$ is either simple or a point curve.
		\par
		We will also use the fact that the free boundary disk flow can be extended continuously to families of curves (see Section 3 of \cite{hass1994}). We will make use of this result in order to apply a Lyusternik--Schnirelmann style argument in Section \ref{sec:tighten}.
		\begin{theorem}
			\label{theorem:family_shortening}
			Let $f:X\to \Omega_{\partial M}M$ be a continuous family of curves parameterized by a compact manifold $X$. There exists a family of maps $f_t:X\to \Omega_{\partial M}M$, $t\in[0,\infty)$, called the image of $f$ under the free boundary disk flow, with the following properties. 
			\begin{enumerate}
				\item 
				The maximum length of a curve in the image of $f_t$ is non-increasing as a function of $t$.
				\item 
				The maximum number of transverse self-intersections of a curve in the image of $f_t$ is non-increasing as a function of $t$.
				\item 
				There is a subsequence of curves $\gamma_t$ in the image of $f_t$ that converges to either an orthogonal geodesic chord or a point curve as $t\to\infty$. Moreover, if any $\gamma_t$ has no transverse self-intersections, then the limit curve has no transverse self-intersections.
				\item 
				Every $f_t$ is homotopic to $f$.\qed
			\end{enumerate}
		\end{theorem} 
		\noindent
		As an alternative, one can use the smooth free boundary flows studied by Ko in \cite{ko2023}. These are flows on compact Riemannian surfaces with convex boundary that are defined (for instance) on families of simple curves with endpoints on the boundary. Under these flows, a simple curve either converges to a simple orthogonal geodesic chord or vanishes in finite time. Moreover, like the disk flows, the flows studied by Ko satisfy the properties required to apply a Lyusternik--Schnirelmann style argument in Section \ref{sec:tighten} (see Section 5 of \cite{ko2023}).
		\par
		To ensure that we obtain a homotopically non-trivial family of curves, we will need some control over the curve shortening processes. We will do so by exploiting convexity. For convenience, we make the following definitions.
		\begin{definition}
			\label{def:convex}
			A closed subset $X\subset M$ is called convex if there exists $\epsilon>0$ so that any pair of points $x,y\in \partial X$ closer than $\epsilon$ are connected by a unique minimizing geodesic which lies in $X$.
			\par
			An oriented simple closed curve is called convex if the closed region it bounds (according to its orientation) is convex.
			\par
			A closed subset $X\subset M$ is called free boundary convex if it is convex in the above sense and there exists $\epsilon>0$ so that any point $x\in X$ closer than $\epsilon$ to $\partial M$ is connected to $\partial M$ by a unique minimizing geodesic which lies in $X$.
			\par 
			An oriented simple curve with endpoints on $\partial M$ is called free boundary convex if the closed region it co-bounds with an arc of $\partial M$ (according to its orientation) is free boundary convex.
		\end{definition} 
		\noindent
		A piecewise-geodesic curve whose interior angles are all at most $\pi$ is convex. If every interior angle this curve makes with $\partial M$ is at most $\pi/2$, then it is free boundary convex. The disk flow and the free boundary disk flow both have convenient properties when applied to a (free boundary) convex initial curve. In such a case, a suitable homotopy between $\gamma_{0}$ and $\gamma_{1}$ (and hence any $\gamma_{i-1}$ and $\gamma_i$) is easily described as follows. Suppose an arc $\eta_0$ of $\gamma_0\cap B_i$ not containing an endpoint of $\gamma_0$ is to be homotoped to an arc $\eta_1$ of $\gamma_1\cap B_i$. For convenience, parameterize $\eta_0$ by the unit interval, and for $t\in[0,1/2]$ define $\rho_t$ to be the minimizing geodesic connecting $\eta_0(1/2-t)$ and $\eta_0(1/2+t)$, noting that $\rho_{1/2}=\eta_1$. 
		Our homotopy will be
		\begin{align*}
			\eta_t= \eta_0|_{[0,1/2-t]}*\rho_{t}*\eta_0|_{[1/2+t,1]}
		\end{align*}
		for $t\in[0,1/2]$. 
		\par
		If instead $\eta_0$ contains an endpoint of $\gamma_0$, we homotope $\eta_0$ to $\eta_1$ through arcs $\eta_t$ defined as follows. For convenience, orient $\eta$ so that $\eta(0)\in\partial B_i$ and $\eta(1)\in\partial M\cap B_i$. The arc $\eta_1$ is the minimizing geodesic connecting $\eta(0)$ to the unique closest point $z$ on $\partial M\cap B_i$, and hence is the unique minimizing geodesic in $B_i$ that connects $\eta(0)$ to $\partial M$ and meets $\partial M$ orthogonally. We will define $\eta_t$ as a minimizing geodesic connecting $\eta(0)$ to a suitable point on $\partial M\cap B_i$ as follows. By convexity, $\eta_0$ meets $\partial M$ at angle at most $\pi/2$. Therefore, by the first variation of energy formula $\eta_0$ is not lengthened by moving $\eta_0(1)$ along $\partial M$ toward $z$. We define $\eta_t$ by moving $\eta_0(1)$ inward until it coincides with $z$, increasing the angle at $\eta_t(1)$ until it equals $\pi/2$. Lastly, we define $\gamma_t$ consecutively in each $B_i$ by applying these homotopies to every arc of $\gamma\cap B_i$. 
		\par
		With this more explicit definition of the flow in mind, we have the following useful result, which shows that convex curves stay convex under the (free boundary) disk flow.
		
		\begin{lemma}
			\label{lemma:shortening_monotonicity}
			Let $\gamma_0$ be a piecewise geodesic (free boundary) convex curve. Let $\gamma_t$ be its image under the (free boundary) disk flow. Then $\gamma_t$ is (free boundary) convex for every $t$, and any two curves $\gamma_s$ and $\gamma_t$, $s\not=t$, do not intersect transversely. If additionally $\gamma_{t_0}$ is a point curve for some $t_0$, but $\gamma_0$ is not, then $\gamma_t$ induces a map of the disk to itself of non-zero local degree.
		\end{lemma}
		\begin{proof}
			We will prove all claims except the last one for $t\in[0,1]$. By induction, the claims will then hold for all $t$. 
			\par
			We first claim that $\gamma_t$ is simple. A pair of minimizing geodesics can intersect at most once and must do so transversely. Therefore, since $\gamma_0$ is simple and hence has no transverse self-intersections, any pair of arcs of $\gamma_0\cap B_i$ that do not contain endpoints of $\gamma_0$ are replaced by a pair of minimizing geodesics that do not intersect. 
			For the same reason, a minimizing geodesic connecting two points of $\gamma_0\cap\partial B_i$ cannot cross a minimizing geodesic connecting a point of $\gamma_0\cap\partial B_i$ to a point on $\partial M$. Moreover, a pair of  minimizing geodesics connecting two points of $\gamma_0\cap\partial B_i$ to their respective nearest points on $\partial M$ cannot cross. Therefore, the endpoint arcs of our homotopy $\gamma_t$, $t\in[0,1]$, never cross each other.
			Thus $\gamma_t$ is simple for $t\in[0,1]$.
			\par 
			We now prove that $\gamma_t$ is convex. Since $\gamma_t$ is simple, we can define $D_t$ as the closed disk bounded by $\gamma_t$ and, if $\gamma_t$ has endpoints on $\partial M$, a portion of $\partial M$. Since $\gamma_t$ is piecewise geodesic, to prove convexity at time $t$ it is sufficient to consider the vertex angles of each $\gamma_t$. Every interior angle needs to be at most $\pi$, and any angles with $\partial M$ need to be at most $\pi/2$. By assumption, this is true of $\gamma_0$. Consider an arc $\eta$ of $\gamma_0\cap B_i$ that does not contain an endpoint of $\gamma_0$. For every $t\in[0,1]$, the corresponding arc of $\gamma_t$ is given by replacing a piece of $\eta$ by a minimizing geodesic connecting its endpoints. Because $\gamma_0$ is convex and each metric ball $B_i$ is totally normal, these minimizing geodesics will lie within $D_0$ and have endpoints on $\partial D_0$. Therefore any new interior angles formed are again at most $\pi$. If we are applying the free boundary disk flow, we must also consider any arc $\eta$ of $\gamma_0$ containing an endpoint of $\gamma_0$. However, our definition of $\gamma_t$ for $t\in[0,1]$ ensures that the angles at the endpoints are always at most $\pi/2$. Thus $\gamma_t$ is (free boundary) convex.
			\par 
			We next claim that $D_{t_2}\subseteq D_{t_1}$ for any $0\leq t_1<t_2\leq 1$, ensuring that $\gamma_{t_1}$ and $\gamma_{t_2}$ do not intersect transversely. By the definition of $\gamma_t$, endpoint-containing arcs move inside of $D_{t}$ at time $t$. Interior arcs of $\gamma_{t_2}$ are obtained by replacing segments of $\gamma_{t_1}$ by minimizing geodesics, which lie in $D_{t_1}$ by convexity. Thus $\gamma_{t_2}$ must lie in $D_{t_1}$, so $D_{t_2}\subseteq D_{t_1}$.
			\par 
			Finally, we prove that if some $\gamma_{t_0}$ is a point and $\gamma_0$ is not, then $\gamma_t$ induces a map of the Euclidean unit disk $D$ to itself of non-zero local degree. Assume that $t_0$ is the first time when $\gamma_t$ is a point. We can define a continuous map $f:D\to D$ by sending the circle of radius $t$, denoted $S(t)$, to $\gamma_{(1-t)t_0}$ for $t\in[0,1]$. We will show that a generic point of $D$ has a single preimage under $f$. Each $\gamma_t$ is simple, so any point in the image of $f|_{S(t)}$ has exactly one preimage in $S(t)$. Suppose some point $x$ lies on both $\gamma_t$ and $\gamma_s$ for $s<t$. Let $\eta^i_s$ be the arc of $\gamma_s\cap B_i$ containing $x$. By the definition of our flows, $x$ can lie on $\gamma_s\cap\gamma_t$ only in two cases.
			In the first case, $\eta^i_s$ is a minimizing geodesic that contains no endpoints of $\gamma_s$.
			In the second case, $\eta_s^i$ contains an endpoint of $\gamma_s$ and is a minimizing geodesic connecting $\eta^i_s\cap\partial B_i$ to the nearest point on $\partial M$. Either way, $x$ lies in the image of $\gamma_{\lfloor s\rfloor}$ by the definition of the free boundary disk flow, since such an arc is either created at an integer time or is fixed by the disk flow in $B_i$.
			Therefore, all such $x$ lie in the union of the images of the curves $\gamma_i$ over all non-negative integers $i$, which is a countable union of closed measure zero subsets of $D$ as claimed.
		\end{proof}
		\noindent
		Note that this lemma shows that a convex curve must flow ``inward"-- in particular, a simple curve bounding a convex region cannot flow outside that region. This will allow us some control over the behaviour of curves under the two flows. Moreover, the fact that simple curves remain simple under the flow will allow us to obtain simple orthogonal geodesic chords in the limit.
		
		\subsection{A Radial Sweepout of $M$}
		\label{sec:family}
		
		We will now construct a radial sweepout $\Gamma_t$.
		We first divide $M$ into free boundary convex disks $\Omega_i$ bounded by pairs of minimizing geodesics using the following variation of Berger's lemma, which is proven in a manner similar to the standard version (see Lemma 4.1 in Chapter 13 Section 4 of \cite{docarmo}). 
		\begin{lemma}
			\label{lemma:berger}
			Let $x$ be a point in $M$ at maximum distance from $\partial M$. Then given any $v\in T_xM$, there exists a minimizing geodesic segment $\gamma$ starting at $x$ and ending on $\partial M$ so that $\langle v,\gamma'(0) \rangle \geq 0$. Moreover, $\gamma$ meets $\partial M$ orthogonally. \qed
		\end{lemma}
		
		\noindent
		We pick a minimal number of these geodesics so that any vector is within $\pi$ of one of our chosen geodesics. 
		We will then contract each pair of these geodesics through curves in $\Sigma_{\partial M} M$. For the most part, we will be able to do so using the free boundary disk flow. However, this flow can get stuck on a simple orthogonal geodesic chord instead of contracting to a point. Because we are interested in producing a \textit{pair} of chords, we will need to apply an additional technique. Therefore, we will try to contract $\partial \Omega_i$ with the standard disk flow. Using the following lemma due to Liokumovich, Nabutovsky and Rotman in \cite{rotman_ls_2017}, we can convert the resulting free homotopy into a homotopy through curves in $\Sigma_{\partial M} M$.
		
		\begin{lemma}[Lemma 3.1 in \cite{rotman_ls_2017}]
			\label{lemma:maeda_alt}
			Suppose $\gamma_0$ is a convex curve that can be contracted to a point $\gamma_1$ through a monotone free homotopy $\gamma_t$ of closed convex curves of length at most $l$. Choose some $p\in \gamma_0$ closest to $\gamma_1$. Then for any $\epsilon>0$, $\gamma_0$ can be homotoped to $p$ through a monotone family of loops based at $p$ of length at most $l+2d(p,\gamma_1)+\epsilon$. Moreover, this new homotopy has non-zero local degree as a map from the disk to itself.\qed
		\end{lemma}
		
		\noindent
		By Lemma \ref{lemma:shortening_monotonicity}, the hypotheses of this lemma are satisfied when $\gamma_t$ is obtained via the disk flow from a convex initial curve $\gamma_0$, such as $\partial \Omega_i$.
		\par
		We can now prove our main lemma, which ensures the existence of either at least one short simple orthogonal geodesic chord or a radial sweepout through short curves using the techniques described above.
		
		\begin{lemma}
			\label{lemma:main}
			Let $M$ be a Riemannian $2$-disk of diameter $d$ with convex boundary of length $P$. Then at least one of the following holds.
			\begin{enumerate}
				\item 
				$M$ admits three simple orthogonal geodesic chords of index zero and length at most $2d$.
				\item 
				$M$ admits a simple orthogonal geodesic chord of index zero and length at most $2d$, and moreover admits $k$ non-simple orthogonal geodesic chords of index zero and lengths at most $4d+k(2d+P)$ for any integer $k\geq 2$.
				\item 
				There exists a radial sweepout $\Gamma_t$ of $M$ such that either
				\begin{enumerate}
					\item 
					$L(\Gamma_t)\leq 3d+P$, or
					\item 
					$M$ admits a simple orthogonal geodesic chord of index zero and length at most $2d$, and $L(\Gamma_t)\leq 7d+2P +\epsilon$ for any $\epsilon>0$.
				\end{enumerate}
			\end{enumerate}
		\end{lemma}
		
		\begin{figure}[ht]
			\centering
			\includegraphics[width=0.45\linewidth]{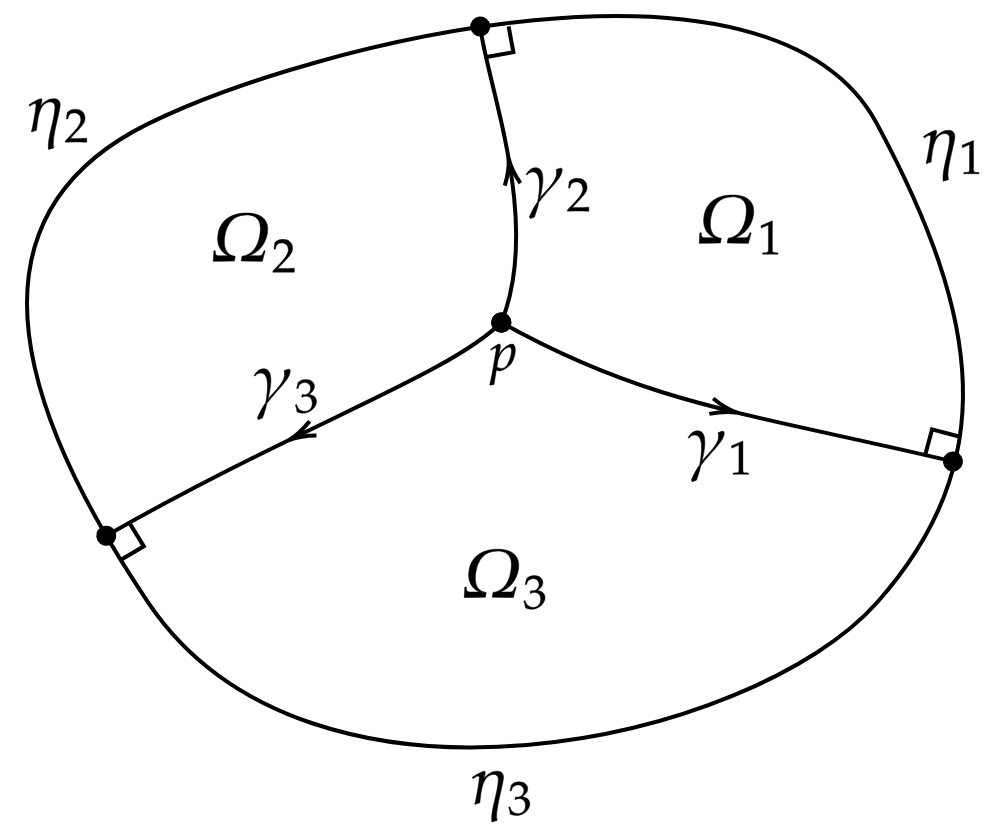}
			\caption{An example illustration of the disks $\Omega_i$ in the proof of Lemma \ref{lemma:main}. Note that there may only be two disks, $\Omega_1$ and $\Omega_2$.}
			\label{fig:berger}
		\end{figure}
		
		\begin{proof}
			Let $p$ be a point in $M$ at maximal distance from $\partial M$. By Lemma \ref{lemma:berger}, there is some minimal number of minimizing geodesic segments $\gamma_i$, $1\leq i$, starting at $p$ and ending on $\partial M$ so that given any $v\in T_pM$, $\langle v,\gamma_i'(0) \rangle \geq 0$ for some $i$. Moreover, each $\gamma_i$ lies in $M$ except at its final point, where it is orthogonal to $\partial M$. If there are only two such geodesics, they necessarily meet at an angle $\pi$ at $p$ and hence form an orthogonal geodesic chord of length at most $2d$. The other possibility is that there are three distinct geodesics $\gamma_i$. For convenience, we number these geodesics cyclically so that $\gamma_3=\gamma_1$ if $i\leq 2$ and otherwise $\gamma_4=\gamma_1$. Divide $\partial M$ into simple arcs $\eta_i$ so that $M$ is divided into closed regions $\Omega_i$ bounded by $\partial\Omega_i=-\gamma_i*\gamma_{i+1}*\eta_i$. 
			Note that the interior angle of $\partial \Omega_i$ at $p$ is at most $\pi$ and the interior angles of $\partial \Omega_i$ at $\partial M$ equal $\pi/2$, making $\Omega_i$ free boundary convex. When there are three geodesics, an example configuration is depicted in Figure \ref{fig:berger}.
			\par 
			Apply the free boundary disk flow to the simple curve $-\gamma_i*\gamma_{i+1}$ for each $i$. By Lemma \ref{lemma:shortening_monotonicity} and Theorem \ref{theorem:free_boundary_flow}, we either obtain a simple orthogonal geodesic chord in $\Omega_i$ of length at most $L(-\gamma_i*\gamma_{i+1})\leq 2d$, or we obtain a monotone homotopy $H_t^i$ in $\Omega_i$ contracting $-\gamma_i*\gamma_{i+1}$ to a point through simple curves with endpoints on $\eta_i\subset\partial M$ such that $L(H_t^i)\leq 2d$. 
			The first case of this lemma occurs when we have three distinct geodesics $\gamma_i$ and all three pairs $-\gamma_i*\gamma_{i+1}$ converge to simple orthogonal geodesic chords. These chords are distinct, as they lie in the interiors of the disjoint regions $\Omega_i$ by convexity. We can also assume that these geodesics have index zero for the following reason. Suppose we obtain a chord $\rho_j$ in $\Omega_j$ with positive index. The curve $\rho_j$ is simple and hence, in conjunction with an arc of $\partial M$, bounds a disk inside $\Omega_j$, which we will call $D_j$. Since $\rho_j$ has positive index, we can perturb it either into $\overline{D_j}$ or into $\overline{\Omega_j\setminus D_j}$ to obtain a strictly shorter curve $\rho_j'$. Continue shortening $\rho_j'$. Since $\rho_j$ is a geodesic that meets $\partial M$ orthogonally, both $\overline{\Omega_j\setminus D_j}$ and $\overline{D_j}$ are free boundary convex. Therefore, either $\rho_j'$ remains in $\overline{\Omega_j\setminus D_j}$ and shortens to a non-trivial orthogonal geodesic chord of length at most $2d$, or it remains in $\overline{D_j}$. In that case, either we once again obtain a chord or we continue our monotone homotopy $H_t^j$ in $D_j$ to finish contracting $-\gamma_j*\gamma_{j+1}$. Note that if we obtain a chord, then it is distinct from $\rho_j$, giving us our desired pair of short orthogonal geodesic chords.
			\par 
			If $-\gamma_i*\gamma_{i+1}$ failed to contract, we apply the standard disk flow to $\partial \Omega_i$. Suppose we thus obtain a simple closed geodesic $\rho$ in $\Omega_i$ of length at most $L(\partial\Omega_i)\leq 2d+P$. Then $\rho$ bounds a convex disk $\Omega_i'\subsetneq \Omega_i$, making $\Omega_i\setminus\Omega_i'$ a free boundary convex annulus. Reparameterize $\rho$ so that $\rho(0)$ is at minimal distance to $\eta_i$ among points on $\rho$, and let $\xi$ be a geodesic in $\Omega_i\setminus\Omega_i'$ connecting $\eta_i$ to $\rho(0)$. Note that such a geodesic exists by convexity of $\Omega_i\setminus\Omega_i'$. Moreover, we can assume $\xi$ has length at most $2d$, since we can connect $\rho(0)$ to the closest point on $\partial\Omega_i$ by a curve of length at most $d$ and, if necessary, follow $\gamma_i$ or $\gamma_{i+1}$ to $\eta_i$.     
			For any $k\geq 2$, consider the curve $\xi*\rho^k*-\xi$ (see Figure \ref{fig:geo_circles}, left). Under the free boundary disk flow, the endpoints of this curve will remain on $\eta_i$ by convexity. However, for $k\geq1$ the curve $\xi*\rho^k*-\xi$ cannot be contracted to a point through curves with endpoints on $\eta_i$. Therefore, by Theorem \ref{theorem:free_boundary_flow} this curve must converge to an orthogonal geodesic chord of length at most $2L(\xi)+kL(\rho)\leq 4d+k(2d+P)$ (see Figure \ref{fig:geo_circles}, right), which is non-simple for $k\geq 2$. As before, we can assume that these chords have index zero or else they can be shortened further. This gives rise to the second case of our lemma. Note that if $k=1$ and $-\gamma_i*\gamma_{i+1}$ is an orthogonal geodesic chord, it is possible for $\xi*\rho^k*-\xi$ to converge to $-\gamma_i*\gamma_{i+1}$. Thus, we require $k\geq 2$ in order to obtain distinct chords.
			\par 
			\begin{figure}[ht]
				\centering
				\includegraphics[width=1\textwidth]{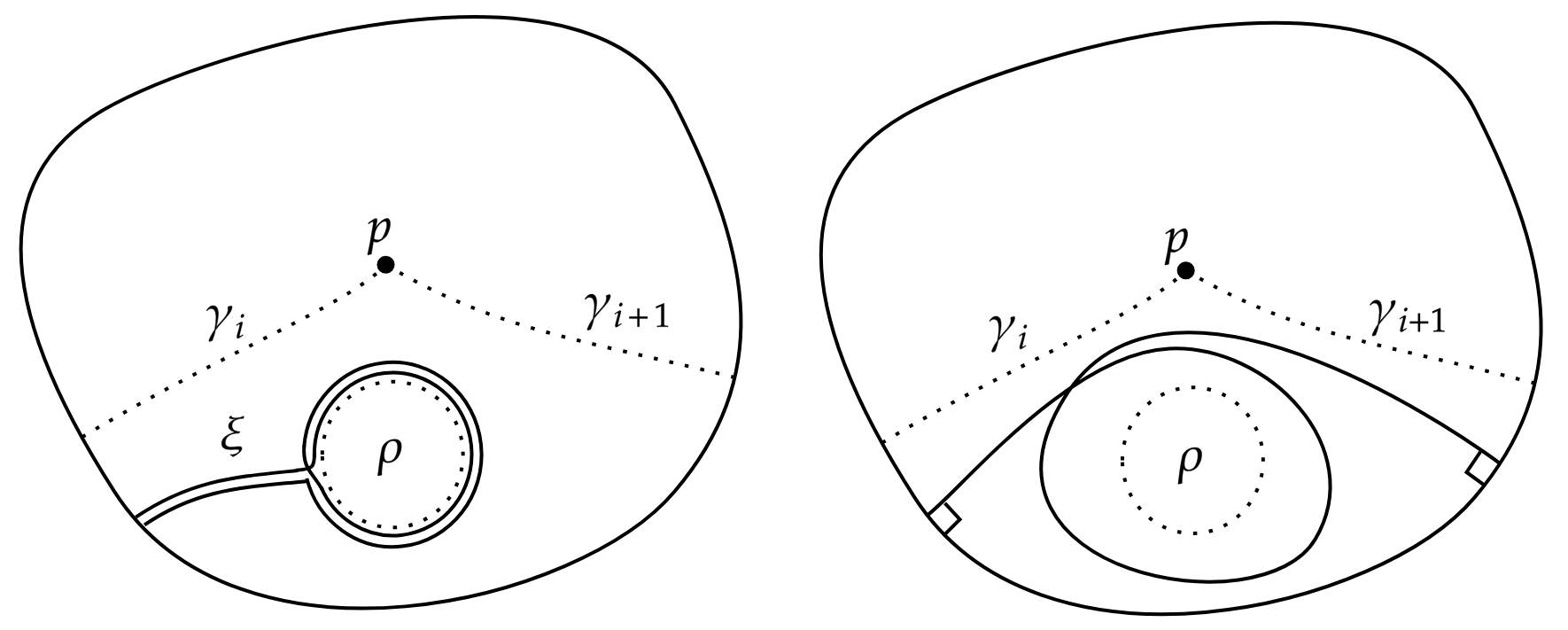}
				\caption{Illustration of the proof of Lemma \ref{lemma:main}. Left: the curve $\xi*\rho^k*-\xi$ (solid) for $k=2$. Right: an orthogonal geodesic chord (solid) obtained by shortening $\xi*\rho^k*-\xi$ for $k=2$.}
				\label{fig:geo_circles}
			\end{figure}
			The other possibility is that the disk flow produces a free homotopy of $\partial \Omega_i$ to a point through curves in $\Omega_i$ of length at most $L(\partial\Omega_i)\leq 2d+P$. By Lemma \ref{lemma:shortening_monotonicity}, this homotopy is monotone and passes through convex curves. Therefore by Lemma \ref{lemma:maeda_alt}, for any $\epsilon>0$ this homotopy can be converted to a monotone homotopy through loops based at some point $q\in\partial \Omega_i$ of length $4d +P+\epsilon$. If $q\not\in\eta_i$, we can add two subarcs of an appropriate choice of $\gamma_i$ to change the basepoint of these loops to an endpoint of $\eta_i$ at the cost of adding at most $2d$ to the length. Note that the resulting homotopy is still monotone. This defines our desired homotopy $H^i_t$, which passes through curves of length at most $6d +P+\epsilon$. 
			\par 
			If we do not encounter a short closed geodesic when applying the disk flow to any $\partial \Omega_i$, then we have successfully defined $H^i_t$ for each $i$. This is case three of our lemma. 
			We use these homotopies to construct $\Gamma_t$ as follows. Define $\eta^i_t$ as the arc of $\eta_i$ connecting $\gamma_i(1)$ to $H_t^i(0)$ (see Figure \ref{fig:third_hom}). Divide the domain $S^1$ of $\Gamma_t$ into two or three equal segments, depending on the number of distinct geodesics $\gamma_i$ we have. We define $\Gamma_t$ on the $i$th segment of $S^1$ by extending $\gamma_i$ to $\gamma_i*\eta^i_0$, then by the family $t\mapsto\gamma_i*\eta_{t}^i*H_t^i$, and then lastly by the contraction of $\gamma_i*-\gamma_i*\gamma_{i-1}$ to $\gamma_{i-1}$ along itself. We have 
			\begin{equation}
				\label{eqn:sweep_bound}
				L(\Gamma_t)
				\leq \max_{i,s}\{L(\gamma_i)+L(\eta^i_s)+L(H_s^i)\}
				\leq d+P+\max_{i,s}\{L(H_s^i)\}.
			\end{equation}
			If every pair $-\gamma_i*\gamma_{i+1}$ contracted to a point under the free boundary disk flow, then $L(H_t^i)\leq 2d$ for every $i$ and hence
			$L(\Gamma_t)\leq 3d+P$. This is case 3(a). Otherwise, at least one curve $-\gamma_i*\gamma_{i+1}$ shortened to a simple orthogonal geodesic chord of length at most $2d$. Therefore, we needed to use the disk flow to contract the boundary of at least one region $\Omega_i$. We have $L(H_t^i)\leq 6d +P+\epsilon$ for such $i$ and hence $L(\Gamma_t)\leq 7d+2P +\epsilon$,
			which is case 3(b). This completes the proof.
		\end{proof}
		\begin{figure}[ht]
			\centering
			\includegraphics[width=0.45\linewidth]{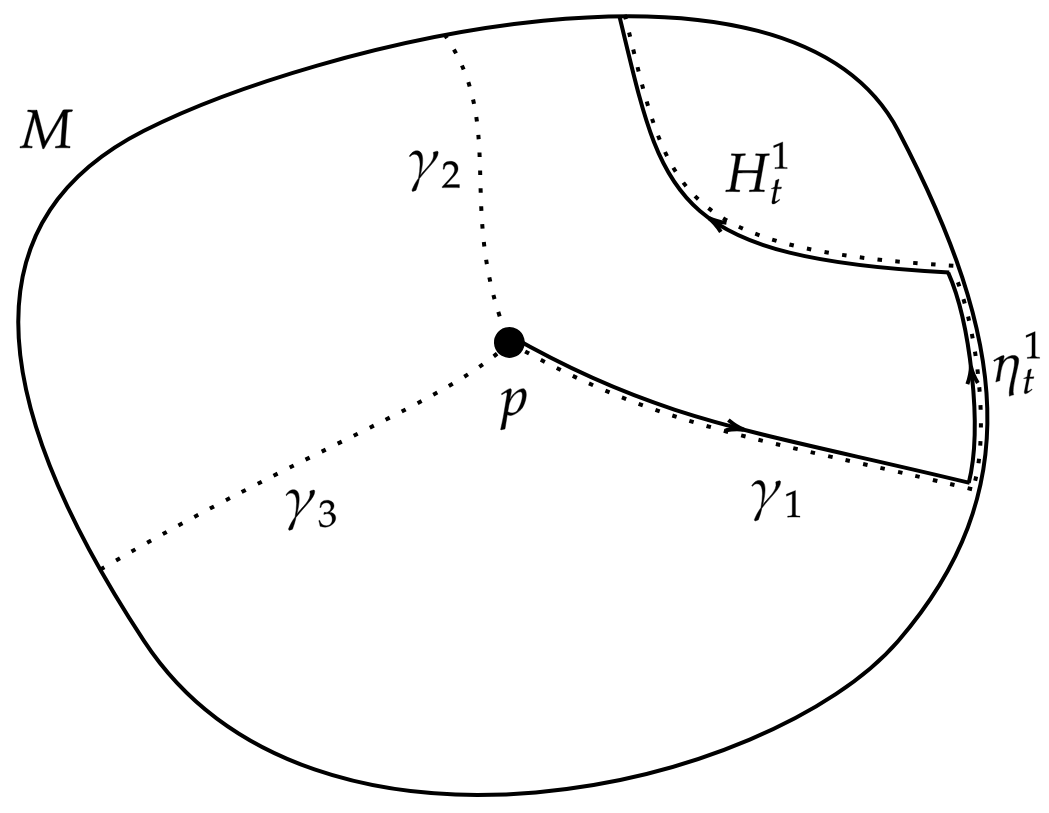}
			\caption{An example of a curve in the family $\Gamma_t$ defined in the proof of Lemma \ref{lemma:main}.}
			\label{fig:third_hom}
		\end{figure}
		
		In fact, by Theorem 1.6 of \cite{lio_2015} there is always a sweepout of $M$ through loops of bounded length regardless of the presence of short closed geodesics or short orthogonal geodesic chords. Unfortunately, this result only guarantees a relatively large length bound, so we avoid applying it whenever possible. 
		
		\begin{theorem}
			\label{theorem:main_alt}
			Let $M$ be a Riemannian $2$-disk with diameter $d$, boundary length $P$, and area $A$. Then for any $\epsilon>0$, there exists a sweepout $\omega_t:[0,1]\to \Sigma_{\partial M} M$ such that 
				$L(\Gamma_t)\leq 2d+2P+686\sqrt{A}+\epsilon.$
		\end{theorem}
		\begin{proof}
			By Theorem 1.6 of \cite{lio_2015}, there is a sweepout $\Gamma_t$ of $M$ through loops based at some $p\in\partial M$ with $L(\Gamma_t)\leq 2d+2P+686\sqrt{A}$. It is implicit in the paper that this sweepout can be taken to pass through simple curves that do not mutually intersect except at $p$. 
			We only require that the sweepout passes through simple curves, which is due to Theorem 1.1 of \cite{chambers_2013}. This theorem states that, for any $\epsilon>0$, a free homotopy that starts and ends at a simple curve and passes through curves of length at most $\ell$ can be converted to a free homotopy through simple curves of length at most $\ell+\epsilon$. This theorem is proven in two steps. First, the authors show that a generic free homotopy consists of curves that only self-intersect transversely and do so at finitely many points in finitely many configurations. They then do surgery on these curves to resolve these self-intersections. These changes are only made in an arbitrarily small neighbourhood of the self-intersections, and so this argument applies equally well to loops based at $p\in \partial M$ that have no self-intersections at $p$. This can be assured by homotoping $\Gamma_t$ slightly to move any self-intersections off of $p$. 
		\end{proof}

		\subsection{Pulling Tight}
		\label{sec:tighten}
		
		Recall that $\Sigma_{\partial M} M$ is the space of unparameterized curves in $M$ with endpoints on $\partial M$ and no transverse self-intersections, and $\Sigma^0_{\partial M} M$ is the space of unparameterized point curves on $\partial M$. To obtain min-max orthogonal geodesic chords, we will construct a pair of maps that realize two non-trivial homology classes of the pair $(\Sigma_{\partial M} M,\Sigma^0_{\partial M} M)$ and whose images consist of curves of bounded length. We will then apply a free boundary curve shortening flow to these maps to obtain a pair of simple orthogonal geodesic chords of bounded length. If these chords are identical, we then apply a Lyusternik--Schnirelmann style argument to obtain infinitely many simple orthogonal geodesic chords of the same length. Further details of this method can be found in Theorem 3.11 of \cite{hass1994} and Appendix A.3 of \cite{klingenberg_1978}.
		
		\begin{lemma}
			\label{lemma:pull_tight}
			Let $M$ be a Riemannian 2-disk with convex boundary that admits a radial sweepout $\Gamma_t$. Define $\max_tL(\Gamma_t)= l$ and $\min_t L(\Gamma_t)=l'$. Then $M$ admits a non-trivial simple orthogonal geodesic chord of length at most $l+l'$ and a second distinct non-trivial simple orthogonal geodesic chord of length at most $2l$.
		\end{lemma}
		\begin{proof}
			We begin by using the radial sweepout to construct 
			two maps realizing our two non-trivial homology classes. The first map 
			$$f_1:([0,\pi],\partial[0,\pi])\to (\Sigma_{\partial M}M,\Sigma^0_{\partial M} M)$$
			is defined as follows. We start by taking $f_1(t)= [-\Gamma_t* \Gamma_{-t}]$ for $t\in[0,\pi]$, where square brackets denote the projection into $\Sigma_{\partial M}M$ (i.e., forgetting the orientation). We must modify this map so that $f_1(0),f_1(\pi)\in \Sigma^0_{\partial M} M$. We thus define $f_1(t)$ for $t\in [-\epsilon,0]$ as the contraction of $[-\Gamma_0* \Gamma_{0}]$ along itself to the point curve $\Gamma_0(1)\in\partial M$. Similarly, for $t\in [\pi,\pi+\epsilon]$ we follow the contraction of $[-\Gamma_\pi* \Gamma_{\pi}]$ along itself to the point curve $\Gamma_\pi(1)\in\partial M$. Then we reparameterize $f_1$ by $[0,\pi]$.
			Our second map is $$f_2:(X,\partial X)\to (\Sigma_{\partial M}M,\Sigma^0_{\partial M} M),$$ 
			where $X$ is the M\"obius band given by a quotient of $[0,\pi]^2$. For $(s,t)\in[0,\pi]^2$, we define $f_2(s,t)=[-\Gamma_{s+t} * \Gamma_{s-t}]$. Since $\Gamma_{2\pi-t}=\Gamma_{-t}$, we have 
			$$f_2(0,t)=[-\Gamma_{t}* \Gamma_{-t}]=[-\Gamma_{-t}* \Gamma_{t}]=[-\Gamma_{\pi+(\pi-t)}* \Gamma_{\pi-(\pi-t)}]=f_2(\pi,\pi-t),$$ 
			which makes $f_2$ into a map from the M\"obius band as claimed. As in the definition of $f_1$, for any $s\in[0,\pi]$ the curves $f_2(s,0)=[-\Gamma_s*\Gamma_s]$ and $f_2(\pi-s,\pi)=[-\Gamma_{-s}*\Gamma_{-s}]$ can be contracted along themselves to point curves on $\partial M$. As before, we modify $f_2$ to incorporate these homotopies for each $s$, so that $f_2(s,0),f_2(\pi-s,\pi)\in\Sigma^0_{\partial M} M$.
			\par
			For $i=1,2$, consider the min-max values
			\begin{align*}
				w_i&=\inf_{\phi\simeq f_i}\max_{x\in \operatorname{dom}f_i} L(\phi(x)),
			\end{align*}
			where the infimum is taken over all maps homotopic to $f_i$. Therefore, $$0\leq w_i\leq \max_{x\in \operatorname{dom}f_i} L(f_i(x))\leq 2l.$$ Note that in fact we can represent $f_1$ by a modification of the map $t\mapsto [-\Gamma_{t_0}*\Gamma_{t}]$, $t\in[t_0,2\pi+t_0]$, for any fixed $t_0\in [0,2\pi]$. This gives the improved estimate $w_1\leq l'+l$ by choosing $t_0$ so that $\Gamma_{t_0}$ has minimum length.
			\par 
			We claim that $w_i>0$. Suppose instead that $w_1=0$. Then there are maps homotopic to $f_1$ whose images consist of curves of arbitrarily small length. Any such map $\widetilde{f_1}$ whose curves are all shorter than the injectivity radius of $M$ can be homotoped so that its image lies in $\Sigma^0_{\partial M} M$ by applying the free boundary disk flow. On the other hand, $\Gamma_t$ is monotone and covers $M$, and hence $f_1$ induces a map of the disk to itself of non-zero local degree. Therefore, $f_1$ induces a homotopically non-trivial map of $(M,\partial M)\simeq S^2$ to itself. However, $f_1$ is homotopic to $\widetilde{f_1}$, which  we just proved is nullhomotopic. This gives rise to a homotopy of a map of non-zero degree to a constant map, which is a contradiction. Similarly, if $w_2=0$, then there would be a contractible map homotopic to $f_2$. Such a map would give rise to a homotopy of $f_1$ to a constant map, since $s\mapsto f_2(s,0)=[-\Gamma_{s+0} * \Gamma_{s-0}]$ equals $f_1$. This once again leads to a contradiction, so $w_i>0$.
			\par 
			Next we claim that there exists an orthogonal geodesic chord of length $w_i$. Such a chord is necessarily non-trivial by the above argument. Given $c\geq 0$, let $\Sigma_{\partial M}^cM$ be the subset of $\Sigma_{\partial M}M$ of curves of length at most $c$. We need the following property of the free boundary flow, adapted from the corresponding statement for the standard disk flow in \cite{hass1994} (see also Theorem 5.1 of \cite{ko2023}). The proof is identical.
			\begin{lemma}[see Lemma 3.12 of \cite{hass1994}]
				\label{lemma:shorten_or_nbhd}
				Let $U_i$ be an arbitrarily small neighbourhood in $\Sigma_{\partial M}M$ of the set of simple orthogonal geodesic chords of length $w_i$. For sufficiently small $\epsilon>0$, the free boundary disk flow maps $\Sigma_{\partial M}^{w_i+\epsilon}M$ into $\Sigma_{\partial M}^{w_i-\epsilon}M\cup U_i$ in unit time. \qed
			\end{lemma}
			\noindent
			We utilize this fact as follows. By the definition of $w_i$, for any $\epsilon>0$ we can pick a map $\widetilde{f_i}$ homotopic to $f_i$ such that the curves in the image of $\widetilde{f_i}$ have length at most $w_i+\epsilon$. If there are no simple orthogonal geodesic chords of length $w_i$, then by the above argument we can apply the flow to the image of $\widetilde{f_i}$ to obtain a homotopic map $\widetilde{f_i'}$ whose image consists of curves shorter than $w_i-\epsilon$ (since we can take $U_i=\emptyset$). This is a contradiction of the definition of $w_i$. Thus there must exist simple orthogonal geodesic chords $\gamma_i$ with $L(\gamma_1)= w_1 \leq l'+l$ and $L(\gamma_2)= w_2\leq 2l$.
			\par 
			If these two curves have different lengths, then being orthogonal geodesic chords they necessarily have distinct images. Suppose instead $w_1=w_2$ and suppose $\gamma_1=\gamma_2$ is the only simple orthogonal geodesic chord of length $w_1$. 
			By the above argument, we can find a map $\widetilde{f_2}$ homotopic to $f_2$ whose image consists of curves in $U_1\cup \Sigma_{\partial M}^{w_1-\epsilon}M$. The set $(\widetilde{f_2})^{-1}(\Sigma_{\partial M}^{w_1-\epsilon}M)\subseteq X$ cannot contain any curves that are non-contractible mod $\partial X$. This is because the restriction of $\widetilde{f_2}$ to such a curve would be a map homotopic to $f_1$ whose image only contains curves strictly shorter than $w_1$, contradicting the definition of $w_1$. Therefore, without loss of generality $(\widetilde{f_2})^{-1}(\Sigma_{\partial M}^{w_1-\epsilon}M)$ consists of some number of disks, and hence its complement must contain a curve that is non-contractible mod $\partial X$. However, the restriction of $\widetilde{f_2}$ to such a curve would be a map homotopic to $f_1$ whose image lies in $U_1$. Therefore, this map is also homotopic to the constant map that sends every point to the curve $\gamma_1$, contradicting the fact that $f_1$ is homotopically non-trivial. Therefore, we must always obtain two distinct simple orthogonal geodesic chords of respective lengths $w_1$ and $w_2$, as claimed.
		\end{proof}
		
		We now combine Lemma \ref{lemma:main}, Theorem \ref{theorem:main_alt} and Lemma \ref{lemma:pull_tight} to prove Theorem \ref{theorem:main}.
		
		\begin{proof}
			Apply Lemma \ref{lemma:main} to $M$. Suppose the third case of this lemma holds, so that we obtain a radial sweepout $\Gamma_t$ of $M$ of length bounded by either $3d+P$ or $7d+2P+\epsilon$. Note that by construction, the shortest curve of $\Gamma_t$ is no longer than $d$. By Lemma
			\ref{lemma:pull_tight}, this sweepout gives rise to a pair of orthogonal geodesic chords. In case 3(a), these chords have respective lengths at most $4d+P$ and $6d+2P$. In case 3(b), they have have respective lengths at most $8d+2P+\epsilon$ and $14d+4P+2\epsilon$. In fact, in case 3(b) we can assume they have respective lengths at most $8d+2P$ and $14d+4P$ by taking a convergent subsequence of chords as $\epsilon$ tends to zero.
			\par
			On the other hand, by Theorem \ref{theorem:main_alt} $M$ always admits a 1-parameter family of curves $\omega_t:[0,1]\to \Sigma_{\partial M}M$ that induces a map of non-zero degree. Moreover, for any $\epsilon$ we can pick $\omega_t$ so that each curve has length at most $2d+2P+686\sqrt{A}+\epsilon$. By applying the proof of Lemma \ref{lemma:pull_tight} to this family and taking a limit as $\epsilon$ tends to zero, we obtain one simple orthogonal geodesic chord with length at most $2d+2P+686\sqrt{A}$. Note that this may be the same chord as one of the pair obtained above if Case 3 of Lemma \ref{lemma:main} holds. Conversely, any chords obtained by applying Lemma \ref{lemma:pull_tight} have positive Morse index, as they were found via the min-max method. Therefore these chords are distinct from any chords of index zero we obtained in the proof of Lemma \ref{lemma:main}, giving us additional chords if Case 1 or Case 2 of Lemma \ref{lemma:main} holds. This completes our proof.
		\end{proof}
		
		\subsection{The Non-Convex Case}
		\label{sec:concave}
		
		We now extend our results to manifolds with non-convex boundary using a technique developed by Seifert \cite{seifert_1949} and utilized by \cite{gluck1984}, proving Theorem \ref{theorem:main_concave}. 
		
		\begin{proof}
			Extend $M$ slightly past $\partial M$ so that it is embedded in a larger Riemannian manifold $M'$. Let $\epsilon>0$ be small enough that we can parameterize an $\epsilon$-neighbourhood $U_\epsilon\subset M'$ of $\partial M$ by minimizing geodesics perpendicular to $\partial M$. This makes $U_\epsilon$ diffeomorphic to $\partial M\times[-\epsilon,\epsilon]$, where $\partial M\times[-\epsilon,0]\subset M$. Let $\gamma_x(t)$ be the minimizing geodesic orthogonal to $\partial M$ at the point $x$. Define $M\cup U_\epsilon=M_\epsilon$ and $N_{t_0}=\partial M\times\{t_0\}$ for $t_0\in [-\epsilon,\epsilon]$.
			\par
			Each geodesic $\gamma_x(t)$ is orthogonal to every set $N_{t_0}$ by the following argument. Identify a small portion of $\partial M$ centred at $x$ with $(-\delta,\delta)$. Let $f(s,t):(-\delta,\delta)\times[0,t_0]\to M$, $0<t_0\leq \epsilon$, be the (non-proper) variation of $\gamma_x(t)$ defined by $f(s,t)=\gamma_s(t)$.
			Since $f(0,t)$ is a geodesic, 
			its first variation of energy with respect to $f$ equals 
			$$
			0=
			\left\langle \frac{df}{ds}(0,t_0),\frac{df}{dt}(0,t_0)\right\rangle
			-
			\left\langle \frac{df}{ds}(0,0),\frac{df}{dt}(0,0)\right\rangle.
			$$
			The second term on the right hand side is zero since $f(0,t)=\gamma_x(t)$ is orthogonal to $N_0=\partial M$. Therefore, the first term on the right hand side must also be zero. This means that $\gamma_x(t_0)$ is orthogonal to $N_{t_0}$ for every $0\leq t_0\leq \epsilon$. The same argument applies for every $-\epsilon\leq t_0\leq 0$ and for every $x\in\partial M$.
			\par 
			Using this fact, the metric on $U_\epsilon$ is of the form $g(x,t)dx^2+dt^2$. We want to alter this metric on $t\in(0,\epsilon]$ so that $\partial M_\epsilon$ is convex. Following \cite{gluck1984}, we can do so by defining a smooth function $\phi(t):[-\epsilon,\epsilon]\to[-\epsilon,\epsilon]$ which is equal to $t$ on $[-\epsilon,\epsilon/3]$ and equal to $\epsilon$ on $[2\epsilon/3,\epsilon]$. We then define the metric on $U_\epsilon$ as $g_\epsilon=g(x,\phi(t))dx^2+dt^2$. Note that this agrees with the original metric on $M\subset M_\epsilon$.
			By Theorem \ref{theorem:main}, we obtain an orthogonal geodesic chord $\gamma$ on $(M_{\epsilon},g_\epsilon)$ of length at most $4\operatorname{diam}_\epsilon(M_{\epsilon}) + L_\epsilon(\partial M_{\epsilon})$ (see Figure \ref{fig:m_epsilon}). Here the subscript denotes that these measurements are made with respect to the metric $g_\epsilon$. 
			\par 
			By uniqueness of geodesics, $\gamma\cap U_\epsilon$ must coincide with some $\gamma_x(t)$ and hence it must cross $\partial M$ orthogonally and enter $M\setminus U_\epsilon$. Consider the arc of $\gamma\cap M$ between the first and second intersection points of $\gamma$ and $\partial M$. This is a geodesic segment in $M$ that is orthogonal to $\partial M$ at at least one endpoint. We want to bound the length of this curve in terms of the geometric parameters of $M$. We have 
			$$
			\operatorname{diam}_\epsilon(M_{\epsilon})\leq d+2\epsilon
			$$
			by the following argument. Consider $x,y\in M_\epsilon$. If $x,y\in M\subset M_\epsilon$, then $d_\epsilon(x,y)\leq d(x,y)\leq d$. If one or both of these points lie in $U_\epsilon\setminus M$, we can connect them to the nearest point in $\partial M$ by a geodesic of length at most $\epsilon$ and then apply the previous case.
			\par
			Next, we need to bound $L_\epsilon(\partial M_{\epsilon})$. We can pick $\epsilon$ small enough and modify the original metric on $U_\epsilon$ so that for any $\delta>0$, $ L(\partial M_{\epsilon})<P+\delta.$
			Therefore, the geodesic we obtained above has length at most
			$$
			4\operatorname{diam}_{\epsilon}(M_{\epsilon}) + L_{\epsilon}(\partial M_{\epsilon})
			\leq 4d +8\epsilon+P+\delta.
			$$
			By taking the limit as $\epsilon,\delta\to 0$, the compactness of $M$ assures that there is in fact such a geodesic of length at most 
			$4d+P.$
		\end{proof}
		\begin{figure}[ht]
			\centering
			\includegraphics[width=0.5\linewidth]{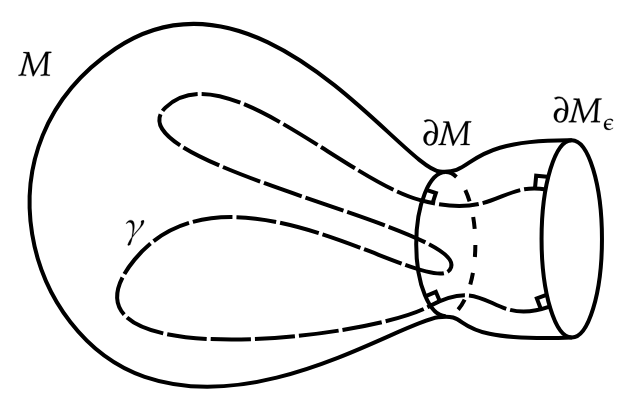}
			\caption{The orthogonal geodesic chord $\gamma$ in the manifold with boundary $M_\epsilon$.}
			\label{fig:m_epsilon}
		\end{figure}
		\noindent
		In the case that $M$ is a disk endowed with the Jacobi metric, we note that applying the limiting procedure detailed in the proof of Theorem A of \cite{gluck1984}-- that is, taking the limit of the chords we obtain as the width $\epsilon$ of the collar $U_\epsilon$ goes to zero-- will give rise to a brake orbit whose length is bounded with respect to the Jacobi metric. Moreover, if $M$ satisfies (for example) the non-resonance condition of \cite{gluck1984}, then the above proof of Theorem \ref{theorem:main_concave} will produce a genuine orthogonal geodesic chord.
		
		\bibliographystyle{abbrv}
		\bibliography{biblio}
		
\end{document}